\documentclass{amsart}

\usepackage{graphicx,epsfig}

\usepackage{amsmath,amssymb,latexsym, amsfonts, amscd, amsthm, xy}

\usepackage{url}

\usepackage{hyperref}

\usepackage{mathrsfs}

\input{xy}
\xyoption{all}

\usepackage[usenames]{color}

\makeindex 

=630pt \headheight = 12pt \headsep = 20pt

\hypersetup{
    colorlinks   = true,
      citecolor    = red
}

\hypersetup{linkcolor=red}

\hypersetup{linkcolor=blue}

\theoremstyle{plain}
\newtheorem{theorem}{Theorem}[section]

\newtheorem{corollary}[theorem]{Corollary}

\newtheorem{lemma}[theorem]{Lemma}

\newtheorem{problem}[theorem]{Problem}

\theoremstyle{definition}
\newtheorem{definition}[theorem]{Definition}
\newtheorem{remark}[theorem]{Remark}

\def\bQ{\mathbb{Q}}
\def\bZ{\mathbb{Z}}

\def\cA{\mathcal{A}}

\def\cI{\mathcal{I}}

\def\cJ{\mathcal{J}}

\def\cM{\mathcal{M}}

\def\cO{\mathcal{O}}

\def\cS{\mathcal{S}}

\def\cX{\mathcal{X}}
\def\cY{\mathcal{Y}}

\def\SL{\mathbf{SL}}

\def\BIG{\mathbf{EBIG}}

\begin{document}

\title[Effectively bounded idempotent generation]{Effectively bounded idempotent generation of certain $2 \times 2$ singular matrices by idempotent matrices over real quadratic number rings}

\author{Dong Quan Ngoc Nguyen}

\date{June 1, 2020}

\address{Department of Applied and Computational Mathematics and Statistics \\
         University of Notre Dame \\
         Notre Dame, Indiana 46556, USA }

\email{\href{mailto:dongquan.ngoc.nguyen@nd.edu}{\tt dongquan.ngoc.nguyen@nd.edu}}

\urladdr{http://nd.edu/~dnguye15}


\subjclass[2020]{15A23, 11R04, 20G30}

\maketitle

\tableofcontents

\section{Introduction}

Let $\cA$ be an integral domain, and let $\cM_n(\cA)$ denote the set of all $n \times n$ matrices with entries in $\cA$. A matrix $M \in \cM_n(\cA)$ is \textit{singular} if the determinant of $M$ is zero in $\cA$. A matrix $M \in \cM_n(\cA)$ is called an \textit{idempotent matrix} if $M^2 = M$. Note that the $n \times n$ identity matrix ${\bf1}_n$ is idempotent. It is obvious that every idempotent matrix $M \ne {\bf1}_n$ is singular. So it is natural to ask whether or not every singular matrix can be written as a product of idempotent matrices, which can be viewed as analogue of finite generation in group theory. Several works have been devoted to this problem. For a field $\cA$, Erdos \cite{E} proved that every singular matrix with entries in $\cA$ is a product of idempotent matrices. For $\cA$ being a division ring or a Euclidean ring, Laffey \cite{laffey1} showed that every singular matrix over $\cA$ is a product of idempotents over $\cA$. Very recently Cossu and Zanardo \cite{cossu-zanardo} considered a similar problem for a certain set of $2 \times 2$ singular matrices over real quadratic number rings; more precisely Cossu and Zanardo proved that if $\cA$ is a real quadratic number ring, then every matrix over $\cA$ of the form $\begin{pmatrix} x & y \\ 0 &0 \end{pmatrix}$ for arbitrary elements $x, y \in \cA$ can be written as a product of idempotents. 

In this paper, we consider the following analogue of bounded generation from group theory in the setting of singular matrices. 

\begin{problem}
\label{main-prob}

Let $\cA$ be an integral domain. Describe the \textit{largest} set of singular matrices over $\cA$, each of whose members can be written as a product of a \textit{bounded} number of idempotent matrices over $\cA$. (For more precise definition of bounded generation in the set of singular matrices, see Definition \ref{def-main}.)

\end{problem}

In contrast to the problem of finite generation of singular matrices by idempotent matrices, there are not many works in literature devoted to studying the above problem. In \cite{E} and \cite{H}, Erdos and Howie showed that every $2 \times 2$ singular matrix with entries from $\bQ$ can be written as a product of two idempotent matrices over $\bQ$. The conclusion no longer holds for $2 \times 2$ singular matrices with entries in $\bZ$. For $n \ge 3$, Laffey \cite{laffey2} proved that every $n \times n$ singular matrix with entries in $\bZ$ can be written as a product of $36n + 217$ idempotent matrices with entries in $\bZ$. Lenders and Xue \cite{LX} improved Laffey's result which shows that every $n \times n$ singular matrix with entries in $\bZ$ can be written as a product of $2n + 1$ idempotent matrices with entries in $\bZ$ for each $n \ge 3$. 

In this paper, we consider Problem \ref{main-prob} for a certain set of $2 \times 2$ singular matrices over quadratic number rings which can be viewed as a natural generalization of the results by Cossu and Zanardo \cite{cossu-zanardo}. More precisely we prove the following result.

\begin{theorem}
\label{thm-introduction}
(see Theorem \ref{thm-main})

Let $\cO_k$ be the ring of integers of a real quadratic number field $k = \bQ(\sqrt{\alpha})$, where $\alpha$ is a positive square-free integer. Then every matrix of the form $\begin{pmatrix} x & y \\ 0 & 0 \end{pmatrix}$ for elements $x, y \in \cO_k$ can be written as a product of at most $15$ idempotent matrices with entries in $\cO_k$. 

\end{theorem}

\begin{remark}

In \cite[Theorem {\bf6.1}]{cohn}, Cohn proved that if $\cO$ is the ring of integers of the imaginary quadratic number field $\bQ(\sqrt{-\alpha})$, where $\alpha$ is a positive square-free integer such that $\alpha \ne 1, 2, 3, 7, 11$, then there exists an invertible $2\times 2$ matrix with entries in $\cO$ that cannot be written as a product of elementary matrices with entries in $\cO$. For such a domain $\cO$, Cossu and Zanardo (see \cite[Proposition {\bf3.4}]{cossu-zanardo1}) showed that there exists a singular $2 \times 2$ matrix with entries in $\cO$ that cannot be written as a product of idempotent matrices. In view of this, we only study Problem \ref{main-prob} the rings of integers of real quadratic number fields $\bQ(\sqrt{\alpha})$, where $\alpha$ is a positive square-free integer.

\end{remark}

Theorem \ref{thm-introduction} is a simplified version of Theorem \ref{thm-main} proved in Section \ref{sec-ebig} which implies that for every matrix of the form $\begin{pmatrix} x & y \\ 0 & 0 \end{pmatrix}$ for elements $x, y \in \cO_k$, there are $19$ \textit{invertible linear transformations} induced by elements in $\SL_2(\cO_k)$ that are needed to convert the original matrix into a product of at most $15$ idempotent matrices with entries in $\cO_k$. Thus Theorem \ref{thm-main} also contains an \textit{effective} algorithm how to convert a matrix of  the form $\begin{pmatrix} x & y \\ 0 & 0 \end{pmatrix}$ for elements $x, y \in \cO_k$ into a product of a bounded number of idempotent matrices. 

The proof of Theorem \ref{thm-introduction} follows the strategy of that of the main theorem of Cossu and Zarnado (see \cite[Theorem {\bf3.2}]{cossu-zanardo}, but there is one key difference between these two proofs. In the proof of Theorem {\bf3.2} in \cite{cossu-zanardo}, Cossu and Zanardo exploited the Euclidean algorithm for $\bZ$ to deduce the fact that for each matrix $\begin{pmatrix} x & y \\ 0 & 0 \end{pmatrix}$ for elements $x, y \in \cO_k$, there exist an \textit{integer} $h \in \bZ$ and an element $\beta \in \cO_k$ such that if $\begin{pmatrix} h & \beta \\ 0 & 0 \end{pmatrix}$ is a product of idempotent matrices, so does $\begin{pmatrix} x & y \\ 0 & 0 \end{pmatrix}$. Since applying the Euclidean algorithm for a couple of integers can generate an arbitrarily long sequence of divisions, one can not obtain a bounded number of transformations needed to convert $\begin{pmatrix} x & y \\ 0 & 0 \end{pmatrix}$ into $\begin{pmatrix} h & \beta \\ 0 & 0 \end{pmatrix}$, which results in a weaker conclusion than Theorem \ref{thm-introduction} in our paper. In order to bound the number of transformations used to convert $\begin{pmatrix} x & y \\ 0 & 0 \end{pmatrix}$ for $x, y \in \cO_k$, into $\begin{pmatrix} h & \beta \\ 0 & 0 \end{pmatrix}$ for some $h \in \bZ$ and $\beta \in \cO_k$, we introduce a new approach in which Dirichlet's theorem on primes in arithmetic progressions will be exploited (see Lemmas \ref{l-main-lemma2} and \ref{l-main-lemma3}.)

The structure of our paper is as follows. In Section \ref{sec-notions}, we introduce some basic notions and notation that will be used throughout the paper. In Section \ref{sec-ebig}, we prove Theorem \ref{thm-introduction} (see Theorem \ref{thm-main})--our main theorem.

\section{Basic notions and notation}
\label{sec-notions}

In this section, we introduce some basic notions and notation which will be used throughout this paper. Throughout this subsection, let $\cA$ denote an integral domain, and let $\cM_2(\cA)$ be the set of all $2 \times 2$ matrices with entries in $\cA$, and let $\SL_2(\cA)$ be the set of all $2 \times 2$ matrices with entries in $\cA$ of determinant $1$. 

For $x, y \in \cA$, denote by $[x \; \; y]$ the matrix $\begin{pmatrix} x & y \\ 0 & 0 \end{pmatrix}$.

For an element $a \in \cA$, write
\begin{align*}
a_{1, 1} &= \begin{pmatrix} a & 1 \\ - 1 & 0 \end{pmatrix}, \\
a_{1, 2} &= \begin{pmatrix} 1 & a \\ 0 & 1 \end{pmatrix},\\
a_{2, 1} &= \begin{pmatrix} 1 & 0 \\ a & 1 \end{pmatrix}, \\
a_{2, 2} &= \begin{pmatrix} 0 & 1 \\ -1 & a \end{pmatrix}.
\end{align*}

\begin{definition}
(transformation $\longrightarrow a_{1, 1}$)

Let $A, B$ be $2 \times 2$ matrices in $\cM_2(\cA)$. We write
\begin{align*}
A \longrightarrow_{a_{1, 1}} B
\end{align*}
for some element $a \in \cA$ if and only if
\begin{align*}
B =  \begin{pmatrix} a & 1 \\ - 1 & 0 \end{pmatrix}^{-1} A  \begin{pmatrix} a & 1 \\ - 1 & 0 \end{pmatrix}.
\end{align*}
That is, $\begin{pmatrix} a & 1 \\ - 1 & 0 \end{pmatrix}$ conjugates $A$ to $B$. 

We also use the notation 
\begin{align*}
A^{a_{1, 1}} = \begin{pmatrix} a & 1 \\ - 1 & 0 \end{pmatrix}^{-1} A  \begin{pmatrix} a & 1 \\ - 1 & 0 \end{pmatrix}.
\end{align*}

It is obvious that $A \longrightarrow_{a_{1, 1}} B$ if and only if
\begin{align*}
B = A^{a_{1, 1}}  = \begin{pmatrix} 0 & -1 \\ 1 & a  \end{pmatrix} A\begin{pmatrix} a & 1 \\ - 1 & 0 \end{pmatrix}.
\end{align*}

\end{definition}

\begin{definition}
(transformation $\longrightarrow a_{1, 2}$)

Let $A, B$ be $2 \times 2$ matrices in $\cM_2(\cA)$. We write
\begin{align*}
A \longrightarrow_{a_{1, 2}} B
\end{align*}
for some element $a \in \cA$ if and only if
\begin{align*}
B = \begin{pmatrix} 1 & a \\ 0 & 1 \end{pmatrix}^{-1} A \begin{pmatrix} 1 & a \\ 0 & 1 \end{pmatrix}.
\end{align*}
That is, $ \begin{pmatrix} 1 & a \\ 0 & 1 \end{pmatrix}$ conjugates $A$ to $B$. 

We also use the notation 
\begin{align*}
A^{a_{1, 2}} = \begin{pmatrix} 1 & a \\ 0 & 1 \end{pmatrix}^{-1} A \begin{pmatrix} 1 & a \\ 0 & 1 \end{pmatrix}.
\end{align*}

It is obvious that $A \longrightarrow_{a_{2, 2}} B$ if and only if
\begin{align*}
B = A^{a_{1, 2}}  = \begin{pmatrix} 1 & -a \\ 0 & 1 \end{pmatrix} A \begin{pmatrix} 1 & a \\ 0 & 1 \end{pmatrix}.\end{align*}

\end{definition}

\begin{definition}
(transformation $\longrightarrow a_{2, 1}$)

Let $A, B$ be $2 \times 2$ matrices in $\cM_2(\cA)$. We write
\begin{align*}
A \longrightarrow_{a_{2, 1}} B
\end{align*}
for some element $a \in \cA$ if and only if
\begin{align*}
B =  \begin{pmatrix} 1  & 0 \\ a & 1 \end{pmatrix}^{-1} A \begin{pmatrix} 1  & 0 \\ a & 1 \end{pmatrix}.
\end{align*}
That is,  $\begin{pmatrix} 1  & 0 \\ a & 1 \end{pmatrix}$ conjugates $A$ to $B$. 

We also use the notation 
\begin{align*}
A^{a_{2, 1}} = \begin{pmatrix} 1  & 0 \\ a & 1 \end{pmatrix}^{-1} A \begin{pmatrix} 1  & 0 \\ a & 1 \end{pmatrix}.
\end{align*}

It is obvious that $A \longrightarrow_{a_{2, 1}} B$ if and only if
\begin{align*}
B = A^{a_{2, 1}}  = \begin{pmatrix} 1  & 0 \\ -a & 1 \end{pmatrix} A \begin{pmatrix} 1  & 0 \\ a & 1 \end{pmatrix}.
\end{align*}

\end{definition}

\begin{definition}
(transformation $\longrightarrow a_{2, 2}$)

Let $A, B$ be $2 \times 2$ matrices in $\cM_2(\cA)$. We write
\begin{align*}
A \longrightarrow_{a_{2, 2}} B
\end{align*}
for some element $a \in \cA$ if and only if
\begin{align*}
B = \begin{pmatrix} 0 & 1 \\ -1 & a \end{pmatrix}^{-1} A \begin{pmatrix} 0 & 1 \\ -1 & a \end{pmatrix}.
\end{align*}
That is, $\begin{pmatrix} 0 & 1 \\ -1 & a \end{pmatrix}$ conjugates $A$ to $B$. 

We also use the notation 
\begin{align*}
A^{a_{2, 2}} = \begin{pmatrix} 0 & 1 \\ -1 & a \end{pmatrix}^{-1} A \begin{pmatrix} 0 & 1 \\ -1 & a \end{pmatrix}.
\end{align*}

It is obvious that $A \longrightarrow_{a_{2, 2}} B$ if and only if
\begin{align*}
B = A^{a_{2, 2}}  = \begin{pmatrix}  a & -1 \\ 1 & 0 \end{pmatrix} A \begin{pmatrix} 0 & 1 \\ -1 & a \end{pmatrix}.
\end{align*}

\end{definition}

The next results are obvious.

\begin{lemma}
\label{l-basic-l1}

Let $x, y$ be elements in an integral domain $\cA$. Then
$$[x \; \; y] \longrightarrow_{0_{2,2}} \begin{pmatrix} 0 & 0 \\ 1 & 0 \end{pmatrix}[-y \; \; x].$$

\end{lemma}

\begin{lemma}
\label{l-basic-l2}

Let $a, b, u$ be elements in an integral domain $\cA$. Then
\begin{align*}
\begin{pmatrix} a & 0 \\ b & 0 \end{pmatrix}^{u_{1, 1}} =\begin{pmatrix} 0 & -1 \\ 1 & u \end{pmatrix} \begin{pmatrix} a & 0 \\ b & 0 \end{pmatrix}\begin{pmatrix} u & 1 \\ -1 & 0 \end{pmatrix} = \begin{pmatrix} -bu &  -b \\ u(a + bu) & a + bu \end{pmatrix}.
\end{align*}

In particular, 
\begin{align*}
\begin{pmatrix} a & 0 \\ 1 & 0 \end{pmatrix}^{(-a)_{1, 1}} = [a \; \; -1],
\end{align*}
and
\begin{align*}
\begin{pmatrix} a & 0 \\ -1 & 0 \end{pmatrix}^{(a)_{1, 1}} = [-a \; \; 1].
\end{align*}
\end{lemma}

\begin{definition}

Let $M$ and $A_1, A_2,\ldots$ be $2 \times 2$ matrices in $\cM_2(\cA)$. For each $n \ge 1$, we inductively define the $2 \times 2$ matrix $M^{A_1A_2\cdots A_n}$ as follows. For $n = 1$, set $M_1 = M^{A_1} = A_1^{-1}M A_1$, and for each $n \ge 2$, 
$$M_n = M^{A_1A_2\cdots A_n} = M_{n -1}^{A_n} = A_n^{-1}M_{n - 1}A_n.$$

\end{definition}

The following result is obvious.

\begin{lemma}
\label{l-sub-main-lemma3}

Let $\cA$ be an integral domain. Let $M_1, \ldots, M_n$ be $2\times 2$ matrices with entries in $\cA$, and let $A_1, \ldots, A_{\ell}$ be $2 \times 2$ matrices in $\SL_2(\cA)$. Then 
\begin{align*}
(M_1 \cdots M_n)^{A_1 \cdots A_{\ell}} = \left(M_1^{A_1 \cdots A_{\ell}}\right) \left(M_2^{A_1 \cdots A_{\ell}}\right)\cdots \left(M_n^{A_1 \cdots A_{\ell}}\right).
\end{align*}

\end{lemma}

The main aim in this paper is to study the following notion for the set of $2 \times 2$ matrices over quadratic number fields.

\begin{definition}
\label{def-main}
(effectively bounded idempotent generation)

Let $\cA$ be an integral domain. A collection of $2 \times 2$ matrices $(M_i)_{i \in I}$ in $\cM_2(\cA)$ is said to admit an \textbf{effectively bounded idempotent generation over $\cA$} if there exist a positive integers $m$ and a nonnegative integer $n$ such that for every $i \in I$, there exist matrices $E_1, \ldots, E_{s_i}$ in $\SL_2(\cA)$ with $s_i \le n$ and there exist idempotent matrices $A_1, \ldots, A_{r_i}$ in $\cM_2(\cA)$ with $r_i \le m$ for which
\begin{align}
\label{e-ebig-def}
M_i^{E_1 \cdots E_{s_i}} = A_1A_2 \cdots A_{r_i}.
\end{align}

Let $\BIG^n_m(\cA)$ denote the \textbf{largest subset of $\cM_2(\cA)$ whose members satisfy the above condition}. Then we can write $(M_i)_{i \in I} \subset \BIG^n_m(\cA)$.

\end{definition}

\begin{remark}

\begin{itemize}

\item []

\item [(i)] Note that each idempotent matrix with entries in $\cA$ is an element in $\BIG^0_1(\cA)$.

\item [(ii)] By Lemma \ref{l-sub-main-lemma3}, if $(M_i)_{i \in I} \subset \BIG^n_m(\cA)$, then it follows from (\ref{e-ebig-def}) that for each $i \in I$,
\begin{align*}
M_i &= (A_1 \cdots A_{r_i})^{E_{s_i}^{-1} \cdots E_1^{-1}} \\
&= \left(A_1^{E_{s_i}^{-1} \cdots E_1^{-1}}\right) \left(A_2^{E_{s_i}^{-1} \cdots E_1^{-1}}\right)\cdots \left(A_{r_i}^{E_{s_i}^{-1} \cdots E_1^{-1}}\right).
\end{align*}

Since each $A_j^{E_{s_i}^{-1} \cdots E_1^{-1}}$ is idempotent and $r_i \le m$, every matrix $M_i$ in the sequence $\{M_i\}_{i \in I}$ can be written as a product of at most $m$ idempotent matrices.

\item [(iii)] Definition \ref{def-main} signifies that after applying at most $n$ invertible linear transformations induced by elements in $\SL_2(\cA)$, every matrix $M_j$ in the sequence $\{M_i\}_{i \in I}$ can be converted into a product of at most $m$ idempotent matrices, where $m$ is independent of the $M_i$.

\end{itemize}

\end{remark}

The next two lemmas are obvious.

\begin{lemma}
\label{l-sub-main-lemma0}

Let $\cA$ be an integral domain. Let $M \in \BIG^n_m(\cA)$ for some positive integers $n, m$, and let $A_1, \ldots, A_{\ell}$ be $2 \times 2$ matrices in $\SL_2(\cA)$. Then 
\begin{align*}
M^{A_1 \cdots A_{\ell}} \in \BIG^{n + \ell}_m(\cA).
\end{align*}

\end{lemma}

\begin{lemma}
\label{l-sub-main-lemma2}

Let $\cA$ be an integral domain, and let $M \in \BIG_{m}^{n}(\cA)$ for some positive integer $m$ and some nonnegative integer $n$. Then $M \in \BIG_{m + r}^{n + s}(\cA)$ for any nonnegative integers $r, s$.

\end{lemma}

\begin{proof}

Note that the identity matrix $\begin{pmatrix} 1 & 0 \\ 0 & 1 \end{pmatrix}$ is an idempotent matrix, and also an element in $\SL_2(\cA)$. Hence the lemma follows immediately.

\end{proof}

\begin{lemma}
\label{l-sub-main-lemma1}

Let $\cA$ be an integral domain. Let $M \in \BIG_{m_1}^{n_1}(\cA)$ and $N \in \BIG_{m_2}^{n_2}(\cA)$, where $m_1, m_2$ are positive integers, and $n_1, n_2$ are nonnegative integers. Then $MN \in \BIG_{m_1 + m_2}^{\min(n_1, n_2)}(\cA)$.

\end{lemma}

\begin{proof}

By assumption, there exist elements $E_1, \ldots, E_{s_1} \in \SL_2(\cA)$ and idempotent matrices $A_1, \ldots, A_{r_1}$, where $s_1, r_1$ are integers with $0 \le s_1 \le n_1$ and $1 \le r_1 \le m_1$ such that
\begin{align*}
M^{E_1\cdots E_{s_1}} = A_1\cdots A_{r_1}.
\end{align*}

Multiplying both sides of the above equation by the product $E_{s_1}^{-1}\cdots E_1^{-1}$, we deduce from Lemma \ref{l-sub-main-lemma3} that
\begin{align}
\label{e1-sub-main-l1}
M &= \left(A_1^{E_{s_1}^{-1} \cdots E_1^{-1}}\right)\left(A_2^{E_{s_1}^{-1} \cdots E_1^{-1}}\right)\cdots \left(A_{r_1}^{E_{s_1}^{-1} \cdots E_1^{-1}}\right) \nonumber \\
&= A_1' A_2' \cdots A_{r_1}',
\end{align}
where 
\begin{align*}
A_i' = A_i^{E_{s_1}^{-1} \cdots E_1^{-1}}
\end{align*}
for each $1 \le i \le r_1$. Note that each $A_i'$ is an idempotent matrix. 

On the other hand, since $N \in \BIG_{m_2}^{n_2}(\cA)$, there exist elements $F_1, \ldots, F_{s_2} \in \SL_2(\cA)$ and idempotent matrices $B_1, \ldots, B_{r_2}$, where $s_2, r_2$ are integers with $0 \le s_2 \le n_2$ and $1 \le r_2 \le m_2$ such that
\begin{align}
\label{e2-sub-main-l1}
N^{F_1\cdots F_{s_2}} = B_1\cdots B_{r_2}.
\end{align}

By (\ref{e1-sub-main-l1}) and (\ref{e2-sub-main-l1}, we deduce from Lemma \ref{l-sub-main-lemma3} that
\begin{align*}
(MN)^{F_1\cdots F_{s_2}} &= M^{F_1\cdots F_{s_2}} N^{F_1\cdots F_{s_2}} \\
&= \left(A_1'^{F_{1} \cdots F_{s_2}}\right)\cdots  \left(A_{r_1}'^{F_{1} \cdots F_{s_2}}\right)B_1\cdots B_{r_2}.
\end{align*}

Since the $A_i'^{F_{1} \cdots F_{s_2}}$ are idempotent, and $r_1 + r_2 \le m_1 + m_2$, $s_2 \le n_2$, the above equation implies that $MN \in \BIG_{m_1 + m_2}^{n_2}(\cA)$. Exchanging the roles of $M, N$, we also obtain that $MN \in \BIG_{m_1 + m_2}^{n_1}(\cA)$, and thus $MN \in \BIG_{m_1 + m_2}^{\min(n_1, n_2)}(\cA)$.

\end{proof}

\section{Effectively bounded idempotent generation} 
\label{sec-ebig}

In this section, we prove our main theorem (see Theorem \ref{thm-main}). We begin by proving several results that will be needed in the proof of our main theorem.

\begin{lemma}
\label{l-main-lemma1}

Let $\cA$ be an integral domain such that there exists a positive integer $n_0$ for which every matrix in $\SL_2(\cA)$ is a product of at most $n_0$ elementary matrices. Set
\begin{align*}
\cS = \{(x, y) \in \cA^2 \; | \; \text{there exist $z, w \in \cA$ such that $\begin{pmatrix} x & y \\ z & w \end{pmatrix} \in \SL_2(\cA)$}\},
\end{align*}
and
\begin{align*}
\cI = \{[x \; \; y] \; | \; (x, y) \in \cS\}
\end{align*}
Then $\cI \subset \BIG_{n_0 + 2}^{2n_0}(\cA)$.

\end{lemma}

\begin{proof}

We will use a similar argument as in the proofs of Lemmas {\bf2.2} and {\bf2.3} in Cossu and Zanardo \cite{cossu-zanardo}. 

Take any $[x \; \; y] \in \cI$. By assumption, there exists $z, w \in \cA$ such that $\begin{pmatrix} x & y \\ z & w \end{pmatrix} \in \SL_2(\cA)$, and thus $xw - yz = 1$. Hence $\begin{pmatrix} x & z \\ y & w \end{pmatrix} \in \SL_2(\cA)$. By assumption, there exist elements $q_0, q_1, \ldots, q_{2n_0 - 1}$ in $\cA$ such that
\begin{align*}
\begin{pmatrix} x & z \\ y & w \end{pmatrix} = \begin{pmatrix} 1 & q_0 \\ 0 & 1 \end{pmatrix}\begin{pmatrix} 1 & 0 \\ q_1 & 1 \end{pmatrix} \cdots \begin{pmatrix} 1 & q_{2n_0 - 2} \\ 0 & 1 \end{pmatrix}\begin{pmatrix} 1 & 0 \\ q_{2n_0 -1} & 1 \end{pmatrix}.
\end{align*}

For each $1 \le i \le 2n_0 - 2$, define 
\begin{align*}
r_{i + 2} = r_i - r_{i + 1} q_{i + 1},
\end{align*}
where we set
\begin{align*}
r_1 &= x - y q_0, \\
r_2 &= y - r_1q_1.
\end{align*}

On setting $r_{-1} = x$, $r_0 = y$, and $r_{2n_0} = 0$, and following the same arguments as in Lemma {\bf2.2} in \cite{cossu-zanardo}, we deduce that
\begin{align}
\label{e1-main-l1}
r_i = q_{i + 1}r_{i + 1} + r_{i + 2}
\end{align}
for each integer $-1 \le i \le 2n_0 - 2$. 

By the above equation, one can verify that for each $-1  \le k \le 2n_0 - 3$, 
\begin{align*}
[r_k \; \; r_{k + 1}]^{(-q_{k + 1})_{2, 1}} = \begin{pmatrix} 1 & 0 \\ q_{k + 1} & 0 \end{pmatrix} [r_{k + 2} \; \; r_{k + 1}],
\end{align*}
and
\begin{align*}
[r_{k + 2} \; \; r_{k + 1}]^{(-q_{k + 2})_{1, 2}} = [r_{k + 2} \; \; r_{k + 3}].
\end{align*}

Thus for each $-1  \le k \le 2n_0 - 3$, we deduce from Lemma \ref{l-sub-main-lemma1} that
\begin{align}
\label{e2-main-l1}
[r_k \; \; r_{k + 1}]^{(-q_{k + 1})_{2, 1} (-q_{k + 2})_{1, 2}} = \begin{pmatrix} 1 & 0 \\ q_{k + 1} & 0 \end{pmatrix}^{(-q_{k + 2})_{1, 2}} [r_{k + 2} \; \; r_{k + 3}].
\end{align}

Applying (\ref{e2-main-l1}) $n_0$ times repeatedly for odd integers $k = -1, 1, 3, \cdots, 2n_0 - 3$, and using Lemma \ref{l-sub-main-lemma1}, we deduce that
\begin{align}
\label{e3-main-l1}
&[x \; \; y ]^{(-q_0)_{2, 1} (-q_{1})_{1, 2} \cdots (-q_{2n_0 - 2})_{2, 1} (-q_{2n_0 - 1})_{1, 2} }  = [r_{-1} \; \; r_{0}]^{(-q_0)_{2, 1} (-q_{1})_{1, 2} \cdots (-q_{2n_0 - 2})_{2, 1} (-q_{2n_0 - 1})_{1, 2} } \nonumber \\
&= \left(\prod_{0 \le h \le n_0 - 1}^{<}\begin{pmatrix} 1 & 0 \\ q_{2h} & 0 \end{pmatrix}^{(-q_{2h + 1})_{1, 2} (-q_{2h + 2})_{2, 1}\cdots (-q_{2n_0 - 2})_{2, 1} (-q_{2n_0 - 1})_{1, 2} }\right)[r_{2n_0 -1} \; \; r_{2n_0}].
\end{align}
(Here the notation $\prod_{i_0 \le h \le j_0}^{<} \alpha_h$ represents the product $\alpha_{i_0}\alpha_{i_0 + 1} \cdots \alpha_{j_0}$ in exactly this ordering of terms $\alpha_h$ appearing in the product.)

Note that 
\begin{align*}
[r_{2n_0 -1} \; \; r_{2n_0}] = [r_{2n_0 - 1} \; \; 0] = \begin{pmatrix} 1 & -1 \\ 0 & 0 \end{pmatrix}\begin{pmatrix} 1 & 0 \\ 1 - r_{2n_0 - 1} & 0 \end{pmatrix}.
\end{align*}

Since $\begin{pmatrix} 1 & -1 \\ 0 & 0 \end{pmatrix}$, $\begin{pmatrix} 1 & 0 \\ 1 - r_{2n_0 - 1} & 0 \end{pmatrix}$, and $\begin{pmatrix} 1 & 0 \\ q_{2h} & 0 \end{pmatrix}^{(-q_{2h + 1})_{1, 2} (-q_{2h + 2})_{2, 1}\cdots (-q_{2n_0 - 2})_{2, 1} (-q_{2n_0 - 1})_{1, 2} }$ are idempotent for each $0 \le h \le n_0 - 1$, we deduce from (\ref{e3-main-l1}) that $[x \; \; y] \in \BIG_{n_0 + 2}^{2n_0}(\cA)$, which verifies the lemma.

\end{proof}

\begin{corollary}
\label{c-main1}

Let $k = \bQ(\sqrt{\alpha})$, where $\alpha$ is a positive square-free integer, and let $\cO_k$ be its ring of integers. Set
\begin{align*}
\cS = \{(x, y) \in \cO_k^2 \; | \; \text{there exist $z, w \in \cO_k$ such that $\begin{pmatrix} x & y \\ z & w \end{pmatrix} \in \SL_2(\cO_k)$}\},
\end{align*}
and
\begin{align*}
\cI = \{[x \; \; y] \; | \; (x, y) \in \cS\}
\end{align*}
Then $\cI \subset \BIG^{18}_{11}(\cO_k)$.

\end{corollary}

\begin{proof}

By Theorem {\bf{1.1}} in Morgan-Rapinchuk-Sury \cite{MRS}, every matrix in $\SL_2(\cO_k)$ is a product of at most $9$ elementary matrices. Hence using Lemma \ref{l-main-lemma1} with $n_0 = 9$, we deduce that $\cI \subset \BIG^{18}_{11}(\cO_k)$.

\end{proof}

\begin{corollary}
\label{c-main2}

Let $\cA$ be an integral domain such that there exists a positive integer $n_0$ for which every matrix in $\SL_2(\cA)$ is a product of at most $n_0$ elementary matrices. Set
\begin{align*}
\cI = \{\begin{pmatrix} x & 0 \\  \pm 1 & 0 \end{pmatrix} \; | \; x \in \cA\}
\end{align*}
Then $\cI \subset \BIG_{n_0 + 2}^{2n_0 + 1}(\cA)$.

\end{corollary}

\begin{proof}

Take an element of the form $\begin{pmatrix} x & 0 \\  1 & 0 \end{pmatrix}$ in $\cI$. By Lemma \ref{l-basic-l2}, we see that
\begin{align*}
\begin{pmatrix} x & 0 \\  1 & 0 \end{pmatrix}^{(-x)_{1, 1}} = [x \; \; -1].
\end{align*}

By Lemma \ref{l-main-lemma1}, and since $\begin{pmatrix} x & - 1 \\ 1 & 0\end{pmatrix} \in \SL_2(\cA)$, $[x \; \; -1]$ belongs in $\BIG_{n_0 + 2}^{2n_0}(\cA)$, and thus $\begin{pmatrix} x & 0 \\  1 & 0 \end{pmatrix} \in \BIG_{n_0 + 2}^{2n_0 + 1}$.

Similarly one can prove that $\begin{pmatrix} x & 0 \\  -1 & 0 \end{pmatrix} \in \BIG_{n_0 + 2}^{2n_0 + 1}$, which implies the lemma immediately.

\end{proof}

From Corollaries \ref{c-main1} and \ref{c-main2}, the following result is obvious.

\begin{corollary}
\label{c-main3}

Let $k = \bQ(\sqrt{\alpha})$, where $\alpha$ is a positive square-free integer, and let $\cO_k$ be its ring of integers. Set
\begin{align*}
\cI = \{\begin{pmatrix} x & 0 \\  \pm 1 & 0 \end{pmatrix} \; | \; x \in \cO_k\}.
\end{align*}
Then $\cI \subset \BIG^{19}_{11}(\cO_k)$.

\end{corollary}

\begin{lemma}
\label{l-main-lemma2}

Let $k = \bQ(\sqrt{\alpha})$, where $\alpha$ is a positive square-free integer such that $\alpha \equiv 2, 3 \pmod{4}$, and let $\cO_k$ be its ring of integers. Let $x, y$ be elements in $\cO_k$. Then there exist an integer $h \in \bZ$ and an element $\beta \in \cO_k$ such that if $[h, \beta] \in \BIG_m^n(\cO_k)$, then $[ x \; \; y ] $ is an element in $\BIG_{m + 24}^{n + 3}(\cO_k)$.

\end{lemma}

\begin{proof}

Let $x = x_1 + x_2\sqrt{\alpha}$, and $y = y_1 + y_2\sqrt{\alpha}$, where $x_1, x_2, y_1, y_2$ are integers. 

Suppose that $x_2 = 0$. Letting $h = x = x_1 \in \bZ$ and $\beta = y \in \cO_k$, we deduce that if $[h, \beta] \in \BIG_m^n(\cO_k)$, then $[ x \; \; y ] $ is an element in $\BIG_{m}^{n}(\cO_k)$, and thus by Lemma \ref{l-sub-main-lemma2}, $[ x \; \; y ] $ is an element in $\BIG_{m + 24}^{n + 3}(\cO_k)$.

Suppose that $y_2 = 0$. We see that
\begin{align*}
[x \; \; y] = [x \; \; y_1] &\longrightarrow_{0_{2, 2}} \begin{pmatrix} 0 & 0 \\ 1 & 0 \end{pmatrix}[-y_1 \; \; x].
\end{align*}

Letting $h = -y_1 \in \bZ$ and $\beta = x \in \cO_k$, we deduce from the above equation that if $[h, \beta] \in \BIG_m^n(\cO_k)$, then $[ x \; \; y ] $ is an element in $\BIG_{m}^{n + 1}(\cO_k)$, and thus by Lemma \ref{l-sub-main-lemma2}, $[ x \; \; y ] $ is an element in $\BIG_{m + 24}^{n + 3}(\cO_k)$.

Suppose that both $x_1, y_1$ are zero. Then 
\begin{align*}
[ x \; \; y ] &= [x_2\sqrt{\alpha} \; \; y_2\sqrt{\alpha}] \\
&= [\sqrt{\alpha} \; \; 0][x_2 \; \; y_2] \\
&=\begin{pmatrix} 1 & - 1 \\ 0 & 0 \end{pmatrix} \begin{pmatrix} 1 & 0 \\ 1 - \sqrt{\alpha} & 0 \end{pmatrix} [x_2 \; \; y_2]\\
\end{align*}

Letting $h = x_2 \in \bZ$ and $\beta = y_2 \in \cO_k$, and since $\begin{pmatrix} 1 & - 1 \\ 0 & 0 \end{pmatrix}$, $\begin{pmatrix} 1 & 0 \\ 1 - \sqrt{\alpha}& 0 \end{pmatrix}$ are idempotent matrices, we deduce that if $[h, \beta] \in \BIG_m^n(\cO_k)$, then $[ x \; \; y ] $ is an element in $\BIG_{m + 2}^{n}(\cO_k)$, and thus by Lemma \ref{l-sub-main-lemma2}, $[ x \; \; y ] $ is an element in $\BIG_{m + 24}^{n + 3}(\cO_k)$.

For the rest of the proof, without loss of generality, we can assume that the following assumptions are true:
\begin{itemize}

\item [(i)] both $x_2, y_2$ are nonzero;

\item [(ii)] at least one of $x_1, y_1$ is nonzero.

\end{itemize}

We consider the following cases.

$\star$ \textit{Case 1. $\gcd(x_1, x_2) = 1$.}

Since $\gcd(x_1, x_2) = 1$, there exist integers $a_0, b_0$ such that
\begin{align*}
a_0 x_1 + b_0 x_2 = 1.
\end{align*}
Thus
\begin{align}
\label{Case1-e1-main-lem1}
ax_1 + bx_2 = -y_2,
\end{align}
where $a = -a_0y_2 \in \bZ$ and $b = -b_0y_2 \in \bZ$. 

We see that
\begin{align*}
[x_1 + x_2\sqrt{\alpha} \; \; y_1 + y_2\sqrt{\alpha}] &\longrightarrow_{(b + a\sqrt{\alpha})_{1, 2}}[x_1 + x_2\sqrt{\alpha} \; \; \; (x_1 + x_2\sqrt{\alpha})(b + a\sqrt{\alpha}) +  y_1 + y_2\sqrt{\alpha}] \\
&= [x_1 + x_2\sqrt{\alpha} \; \; \; (bx_1 + ax_2\alpha + y_1) +  (ax_1 + bx_2 + y_2)\sqrt{\alpha}] \\
&= [x_1 + x_2\sqrt{\alpha} \; \; \; bx_1 + ax_2\alpha + y_1] \; (\text{see (\ref{Case1-e1-main-lem1})}) \\
&\longrightarrow_{0_{2, 2}} \begin{pmatrix} 0 & 0 \\ 1 & 0 \end{pmatrix} [-(bx_1 + ax_2\alpha + y_1) \; \; \; x_1 + x_2\sqrt{\alpha}].
\end{align*}

By Corollary \ref{c-main3}, $\begin{pmatrix} 0 & 0 \\ 1 & 0 \end{pmatrix}  \in \BIG_{11}^{19}(\cO_k)$. Since $h = -(bx_1 + ax_2\alpha + y_1) \in \bZ$ and $\beta = x_1 + x_2\sqrt{\alpha}$ is an element in $\cO_k$, we deduce from Lemma \ref{l-sub-main-lemma1} that if $[h, \beta] \in \BIG_m^n(\cO_k)$, then $[x \; \; y] = [x_1 + x_2\sqrt{\alpha} \; \; y_1 + y_2\sqrt{\alpha}]  \in \BIG_{m + 11}^{\min(19, n) + 2}(\cO_k)$.

$\star$ \textit{Case 2. $\gcd(y_1, y_2) = 1$.}

We see that
\begin{align*}
[x_1 + x_2\sqrt{\alpha} \; \; y_1 + y_2\sqrt{\alpha}] &\longrightarrow_{0_{2, 2}} \begin{pmatrix} 0 & 0 \\ 1 & 0 \end{pmatrix}[-y_1 - y_2\sqrt{\alpha} \; \; x_1 + x_2\sqrt{\alpha}].
\end{align*}

Since $\gcd(-y_1, -y_2) = 1$, using the result in \textit{Case 1} with $x_1 + x_2\sqrt{\alpha}$ replaced by $-y_1 - y_2\sqrt{\alpha}$, we deduce that there exists an integer $h \in \bZ$ and an element $\beta \in \cO_k$ such that if $[h, \beta] \in \BIG_m^n(\cO_k)$, then $[-y_1 - y_2\sqrt{\alpha} \; \; x_1 + x_2\sqrt{\alpha}]  \in \BIG_{m + 11}^{\min(19, n) + 2}(\cO_k)$. By Corollary \ref{c-main3}, $\begin{pmatrix} 0 & 0 \\ 1 & 0 \end{pmatrix}  \in \BIG_{11}^{19}(\cO_k)$. Thus it follows from the above equation and Lemma \ref{l-sub-main-lemma1} that there exists an integer $h \in \bZ$ and an element $\beta \in \cO_k$ such that if $[h, \beta] \in \BIG_m^n(\cO_k)$, then $[x \; \; y]$ is an element in $\BIG_{m + 22}^{\min(19, \min(19, n) + 2)}(\cO_k)$ which is equivalent to $[x \; \; y] \in \BIG_{m + 22}^{\min(19, n + 2)}(\cO_k)$.

$\star$ \textit{Case 3. $s = \gcd(x_1, x_2) > 1$ and $r = \gcd(y_1, y_2) >1$.}

$\bullet$ \textit{Subcase 3A. $\gcd(s, r) = 1$.}

Set $\lambda = \gcd(x_1, y_1)$, and $\epsilon = \gcd(x_2, y_2)$. Since $x_2, y_2$ are nonzero, and at least one of $x_1, y_1$ is nonzero, $\lambda, \epsilon$ are positive integers.

Write $x_1 = \lambda z_1$, $y_1 = \lambda w_1$, $x_2 = \epsilon z_2$, and $y_2 = \epsilon w_2$, where $z_1, z_2, w_1, w_2$ are all integers such that $\gcd(z_1, w_1) = \gcd(z_2, w_2) = 1$. 

Assume first that $z_1w_2 - z_2w_1 = 0$. 

We see that $z_1w_2 = z_2w_1$. Since $x_2, y_2$ are nonzero and at least one of $x_1, y_1$ is nonzero, the last identity implies that all of $z_1, z_2, w_1, w_2$ are nonzero.

Since $\gcd(z_1, w_1) = 1$ and $z_1$ divides $z_2w_1$, we deduce that $z_1$ divides $z_2$. On the other hand, since $\gcd(z_2, w_2) = 1$ and $z_2$ divides  $z_1w_2$, it follows that $z_2$ divides $z_1$. Thus $z_1 = \delta z_2$, where $\delta \in \{\pm 1\}$. Thus $w_1 = \delta w_2$. Therefore
\begin{align*}
[x \; \; y] &= [\lambda z_1 + \epsilon z_2\sqrt{\alpha} \; \; \; \lambda w_1 + \epsilon w_2\sqrt{\alpha}] \\
&=  [z_2(\lambda\delta + \epsilon \sqrt{\alpha}) \; \; \; w_2(\lambda \delta + \epsilon \sqrt{\alpha})] \\
&= [\lambda\delta + \epsilon \sqrt{\alpha} \; \; \; 0][z_2 \; \; w_2].
\end{align*}

We see that $s = \gcd(x_1, x_2) = \gcd(\lambda \delta z_2, \epsilon z_2) = |z_2|\gcd(\lambda \delta, \epsilon)$ and $r = \gcd(y_1, y_2) = \gcd(\lambda \delta w_2, \epsilon w_2) = |w_2|\gcd(\lambda \delta, \epsilon)$. Since $\gcd(s, r) = 1$ and $\gcd(z_2, w_2) = 1$, we deduce that $\gcd(\lambda \delta, \epsilon) = 1$. 

Since $\gcd(z_2, w_2) = 1$, Corollary \ref{c-main1} implies that $[z_2 \; \; w_2]$ is in $\BIG_{11}^{18}(\cO_k)$. By the result in \textit{Case 1} with $x_1, x_2$ replaced by $\lambda \delta, \epsilon$, respectively, we deduce that there exist an integer $h \in \bZ$ and an element $\beta \in \cO_k$ such that if $[h \; \; \beta] \in \BIG_m^n(\cO_k)$, then $[\lambda\delta + \epsilon \sqrt{\alpha} \; \; \; 0] \in \BIG_{m + 11}^{\min(19, n) + 2}(\cO_k)$. Using Lemma \ref{l-sub-main-lemma1}, we deduce that there exist an integer $h \in \bZ$ and an element $\beta \in \cO_k$ such that if $[h \; \; \beta] \in \BIG_m^n(\cO_k)$, then $[x \; \; y] \in \BIG_{m + 22}^{\min(18, \min(19, n) + 2)}(\cO_k)$ which is equivalent to $[x \; \; y] \in \BIG_{m + 22}^{\min(18, n + 2)}(\cO_k)$.

Assume now that $z_1w_2 - z_2w_1 \ne 0$. 

Set 
\begin{align*}
\cI = \{\text{primes $\ell$ such that $z_1w_2 - z_2w_1 \equiv 0 \pmod{\ell}$}\}.
\end{align*}
Note that $\cI$ is a finite nonempty set. 

Note that the assumption implies that $\lambda, \epsilon$ are relatively prime, and thus at least one of them is odd. 

Suppose first that $\lambda$ is odd, i.e., every prime factor of $\lambda$ is odd. 

Set
\begin{align*}
\cJ_{\lambda} = \{\text{primes $\ell$ such that $\lambda \equiv 0 \pmod{\ell}$ and $z_2 \not\equiv 0 \pmod{\ell}$}\}.
\end{align*}

Note that all primes in $\cJ_{\lambda}$, if any, are odd. If $\cJ_{\lambda} \ne \emptyset$, write $\cJ_{\lambda} = \cX_{\lambda} \cup \cY_{\lambda}$, where 
\begin{align*}
\cX_{\lambda} = \{\text{primes $\ell$ in $\cJ_{\lambda}$ such that $z_1 \equiv 0 \pmod{\ell}$}\},
\end{align*}
and
\begin{align*}
\cY_{\lambda} = \{\text{primes $\ell$ in $\cJ_{\lambda}$ such that $z_1 \not\equiv 0 \pmod{\ell}$}\}.
\end{align*}

Note that $\cX_{\lambda} \cap \cY_{\lambda} = \emptyset$. Since $\ell$ is odd for every prime $\ell$ in $\cY_{\lambda}$, one can choose, for each prime $\ell \in \cY_{\lambda}$, an integer $b_{\ell}$ such that $b_{\ell} \not\equiv -z_1^{-1}w_1 \pmod{\ell}$ and $b_{\ell} \not\equiv -z_2^{-1}w_2 \pmod{\ell}$. For each prime $\ell \in \cX_{\lambda}$, choose an integer $a_{\ell}$ such that $a_{\ell} \not\equiv -z_2^{-1}w_2 \pmod{\ell}$. By the Chinese Remainder Theorem, there exists an integer $u_{\lambda}$ such that
\begin{align}
\label{e1-main-lem1}
\begin{cases}
u_{\lambda} \equiv a_{\ell} \pmod{\ell} \; \; \text{for every prime $\ell \in \cX_{\lambda}$}, \\
u_{\lambda} \equiv b_{\ell} \pmod{\ell} \; \; \text{for every prime $\ell \in \cY_{\lambda}$}.
\end{cases}
\end{align}

Note that if exactly one of $\cX_{\epsilon}$ and $\cY_{\epsilon}$ is empty, $u_{\epsilon}$ is chosen so as to satisfy exactly one of the above congruence conditions which corresponds to the nonempty set.  

Set 
\begin{align*}
P_{\lambda} = \prod_{\ell \in \cJ_{\lambda}} \ell.
\end{align*}

If $\cJ_{\lambda} = \emptyset$, we set $P_{\lambda} = u_{\lambda} = 1$. We claim that $\gcd(z_1P_{\lambda}, z_1u_{\lambda} + w_1) = 1$. Since $\gcd(z_1, w_1) = 1$, it is clear that $\gcd(z_1, z_1u_{\lambda} + w_1) = 1$. By the choice of $u_{\lambda}$, it is also clear that $\gcd(P_{\lambda}, z_1u_{\lambda} + w_1) = 1$, and thus
\begin{align*}
\gcd(z_1P_{\lambda}, z_1u_{\lambda} + w_1) = 1.
\end{align*}

By Dirichlet's theorem on primes in arithmetic progressions, there exist infinitely many integers $f$ such that $p = (z_1P_{\lambda})f +z_1u_{\lambda} + w_1 $ is a prime for which $p \not\in \cI$ and $\gcd(p, \epsilon) = 1$. Take such an integer $f$, and set $p = (z_1P_{\lambda})f +z_1u_{\lambda} + w_1 = z_1e + w_1$, where $e = P_{\lambda}f  + u_{\lambda}$. 

We see that
\begin{align}
\label{Subcase3A-e2-main-lem1}
[ x_1 + x_2\sqrt{\alpha} \; \; y_1 + y_2\sqrt{\alpha} ] \longrightarrow_{e_{1, 2}} &[x_1 + x_2\sqrt{\alpha} \; \; \;  x_1e + y_1 + (x_2e + y_2)\sqrt{\alpha}]  \nonumber \\
&=[x_1 + x_2\sqrt{\alpha} \; \; \;  \lambda(z_1e + w_1) + \epsilon(z_2e + w_2)\sqrt{\alpha}] \nonumber \\
&= [x_1 + x_2\sqrt{\alpha} \; \; \;  p\lambda + \epsilon(z_2e + w_2)\sqrt{\alpha}] \nonumber \\
&= [x_1 + x_2\sqrt{\alpha} \; \; \; y_1' +y_2'\sqrt{\alpha}], 
\end{align}
where
\begin{align*}
y_1' &= p\lambda, \\
y_2' &= \epsilon(z_2e + w_2).
\end{align*} 

We contend that $\gcd(y_1', y_2') = 1$. Indeed, we first prove that $p$ does not divide $z_2e + w_2$. Assume the contrary, i.e., $p = z_1e + w_1 \equiv 0 \pmod{p}$ and $z_2e + w_2 \equiv 0 \pmod{p}$. Thus $z_1w_2 - z_2w_1 \equiv 0 \pmod{p}$, which implies that $p \in \cI$, a contradiction to the choice of $p$. Thus $\gcd(p, z_2e + w_2) = 1$. By the choice of $p$, $\gcd(p, \epsilon) = 1$, and thus
\begin{align}
\label{Subcase3A-e3-main-lem1}
\gcd(p, \epsilon(z_2e + w_2)) = 1.
\end{align}

We now prove that $\gcd(\lambda, z_2e + w_2) = 1$. Assume the contrary, i.e., there exists a prime factor $\ell$ of $\lambda$ such that $z_2 e + w_2 \equiv 0 \pmod{\ell}$. If $z_2 \equiv 0 \pmod{\ell}$, then it follows that $w_2 \equiv 0 \pmod{\ell}$, which is a contradiction since $\gcd(z_2, w_2) = 1$. Hence $z_2 \not\equiv 0 \pmod{\ell}$, and thus $\ell \in \cJ_{\lambda}$. 

Recall that $e = P_{\lambda} f + u_{\lambda}$ and since $\ell$ divides $P_{\lambda}$, we deduce that
\begin{align*}
z_2e + w_2 = z_2 u_{\lambda} + w_2 \equiv 0 \pmod{\ell}.
\end{align*}
Thus $u_{\lambda} \equiv - z_2^{-1}w_2 \pmod{\ell}$, which is a contradiction to the choice of $u_{\lambda}$. Thus $\gcd(\lambda, z_2e + w_2) = 1$. Since $\gcd(\lambda, \epsilon) = 1$, we deduce that
\begin{align}
\label{Subcase3A-e4-main-lem1}
\gcd(\lambda, \epsilon(z_2e + w_2)) = 1.
\end{align}

By (\ref{Subcase3A-e3-main-lem1}), (\ref{Subcase3A-e4-main-lem1}), we deduce that $\gcd(y_1', y_2') =  \gcd(p\lambda, \epsilon(z_2e + w_2)) = 1$. From (\ref{Subcase3A-e2-main-lem1}), we can use the result in \textit{Case 2} for $(y_1', y_2')$ in place of $(y_1, y_2)$ to deduce that there exists an integer $h \in \bZ$ and an element $\beta \in \cO_k$ such that if $[h, \beta] \in \BIG_m^n(\cO_k)$, then $[x_1 + x_2\sqrt{\alpha} \; \; \; y_1' +y_2'\sqrt{\alpha}]$ is an element in $\BIG_{m + 22}^{\min(19, n + 2)}(\cO_k)$. It follows from (\ref{Subcase3A-e2-main-lem1}) that  that there exists an integer $h \in \bZ$ and an element $\beta \in \cO_k$ such that if $[h, \beta] \in \BIG_m^n(\cO_k)$, then $[ x_1 + x_2\sqrt{\alpha} \; \; y_1 + y_2\sqrt{\alpha} ] $ is an element in $\BIG_{m + 22}^{\min(20, n + 3)}(\cO_k)$

Suppose now that $\epsilon$ is odd, i.e., every prime factor of $\epsilon$ is odd. We use a similar argument as above with $\epsilon$ in place of $\lambda$. 

Set
\begin{align*}
\cJ_{\epsilon} = \{\text{primes $\ell$ such that $\epsilon \equiv 0 \pmod{\ell}$ and $z_1 \not\equiv 0 \pmod{\ell}$}\}.
\end{align*}

Note that all primes in $\cJ_{\epsilon}$, if any, are odd. If $\cJ_{\epsilon} \ne \emptyset$, write $\cJ_{\epsilon} = \cX_{\epsilon} \cup \cY_{\epsilon}$, where 
\begin{align*}
\cX_{\epsilon} = \{\text{primes $\ell$ in $\cJ_{\epsilon}$ such that $z_2 \equiv 0 \pmod{\ell}$}\},
\end{align*}
and
\begin{align*}
\cY_{\epsilon} = \{\text{primes $\ell$ in $\cJ_{\epsilon}$ such that $z_2 \not\equiv 0 \pmod{\ell}$}\}.
\end{align*}

Since $\ell$ is odd for every prime $\ell$ in $\cY_{\epsilon}$, one can choose, for each prime $\ell \in \cY_{\epsilon}$, an integer $b_{\ell}$ such that $b_{\ell} \not\equiv -z_1^{-1}w_1 \pmod{\ell}$ and $b_{\ell} \not\equiv -z_2^{-1}w_2 \pmod{\ell}$. For each prime $\ell \in \cX_{\epsilon}$, choose an integer $a_{\ell}$ such that $a_{\ell} \not\equiv -z_1^{-1}w_1 \pmod{\ell}$. By the Chinese Remainder Theorem, there exists an integer $u_{\epsilon}$ such that
\begin{align}
\label{e1-main-lem1}
\begin{cases}
u_{\lambda} \equiv a_{\ell} \pmod{\ell} \; \; \text{for every prime $\ell \in \cX_{\epsilon}$}, \\
u_{\lambda} \equiv b_{\ell} \pmod{\ell} \; \; \text{for every prime $\ell \in \cY_{\epsilon}$}.
\end{cases}
\end{align}

Note that if exactly one of $\cX_{\epsilon}$ and $\cY_{\epsilon}$ is empty, $u_{\epsilon}$ is chosen so as to satisfy exactly one of the above congruence conditions which corresponds to the nonempty set.  

Set 
\begin{align*}
P_{\epsilon} = \prod_{\ell \in \cJ_{\epsilon}} \ell.
\end{align*}

If $\cJ_{\epsilon} = \emptyset$, we set $P_{\epsilon} = u_{\epsilon} = 1$. We claim that $\gcd(z_2P_{\epsilon}, z_2u_{\epsilon} + w_2) = 1$. Since $\gcd(z_2, w_2) = 1$, it is clear that $\gcd(z_2, z_2u_{\epsilon} + w_2) = 1$. By the choice of $u_{\epsilon}$, it is also clear that $\gcd(P_{\epsilon}, z_2u_{\epsilon} + w_2) = 1$, and thus
\begin{align*}
\gcd(z_2P_{\epsilon}, z_2u_{\epsilon} + w_2) = 1.
\end{align*}

By Dirichlet's theorem on primes in arithmetic progressions, there exist infinitely many integers $f$ such that $p = (z_2P_{\epsilon})f + z_2u_{\epsilon} + w_2 $ is a prime for which $p \not\in \cI$ and $\gcd(p, \lambda) = 1$. Take such an integer $f$, and set $p = (z_2P_{\epsilon})f + z_2u_{\epsilon} + w_2 = z_2e + w_2$, where $e = P_{\epsilon}f  + u_{\epsilon}$.

We see that
\begin{align}
\label{Subcase3A-e5-main-lem1}
[ x_1 + x_2\sqrt{\alpha} \; \; y_1 + y_2\sqrt{\alpha} ] \longrightarrow_{e_{1, 2}} &[x_1 + x_2\sqrt{\alpha} \; \; \;  x_1e + y_1 + (x_2e + y_2)\sqrt{\alpha}]  \nonumber \\
&=[x_1 + x_2\sqrt{\alpha} \; \; \;  \lambda(z_1e + w_1) + \epsilon(z_2e + w_2)\sqrt{\alpha}] \nonumber \\
&= [x_1 + x_2\sqrt{\alpha} \; \; \;  \lambda(z_1e + w_1) + + p\epsilon\sqrt{\alpha}] \nonumber \\
&= [x_1 + x_2\sqrt{\alpha} \; \; \; y_1' +y_2'\sqrt{\alpha}], 
\end{align}
where
\begin{align*}
y_1' &=  \lambda(z_1e + w_1), \\
y_2' &= p\epsilon.
\end{align*} 


We contend that $\gcd(y_1', y_2') = 1$. Indeed, we first prove that $p$ does not divide $z_1e + w_1$. Assume the contrary, i.e., $p$ divides $z_1e + w_1$, and thus $p = z_2e + w_2 \equiv 0 \pmod{p}$ and $z_1e + w_1 \equiv 0 \pmod{p}$. Thus $z_1w_2 - z_2w_1 \equiv 0 \pmod{p}$, which implies that $p \in \cI$, a contradiction to the choice of $p$. Thus $\gcd(p, z_1e + w_1) = 1$. By the choice of $p$, $\gcd(p, \lambda) = 1$, and thus
\begin{align}
\label{Subcase3A-e6-main-lem1}
\gcd(p, \lambda(z_1e + w_1)) = 1.
\end{align}

We now prove that $\gcd(\epsilon, z_1e + w_1) = 1$. Assume the contrary, i.e., there exists a prime factor $\ell$ of $\epsilon$ such that $z_1 e + w_1 \equiv 0 \pmod{\ell}$. If $z_1 \equiv 0 \pmod{\ell}$, then it follows that $w_1 \equiv 0 \pmod{\ell}$, which is a contradiction since $\gcd(z_1, w_1) = 1$. Hence $z_1 \not\equiv 0 \pmod{\ell}$, and thus $\ell \in \cJ_{\epsilon}$. 

Recall that $e = P_{\epsilon} f + u_{\epsilon}$. Since $\ell$ divides $P_{\epsilon}$, we deduce that
\begin{align*}
z_1e + w_1 = z_1 u_{\epsilon} + w_1 \equiv 0 \pmod{\ell}.
\end{align*}
Thus $u_{\epsilon} \equiv - z_1^{-1}w_1 \pmod{\ell}$, which is a contradiction to the choice of $u_{\epsilon}$. Thus $\gcd(\epsilon, z_1e + w_1) = 1$. Since $\gcd(\lambda, \epsilon) = 1$, we deduce that
\begin{align}
\label{Subcase3A-e7-main-lem1}
\gcd(\epsilon, \lambda(z_1e + w_1)) = 1.
\end{align}

By (\ref{Subcase3A-e6-main-lem1}), (\ref{Subcase3A-e7-main-lem1}), we deduce that $\gcd(y_1', y_2') =  \gcd(\lambda(z_1e + w_1), p\epsilon) = 1$. From (\ref{Subcase3A-e5-main-lem1}), we can use the result in \textit{Case 2} for $(y_1', y_2')$ in place of $(y_1, y_2)$ to deduce that there exists an integer $h \in \bZ$ and an element $\beta \in \cO_k$ such that if $[h, \beta] \in \BIG_m^n(\cO_k)$, then $[x_1 + x_2\sqrt{\alpha} \; \; \; y_1' +y_2'\sqrt{\alpha}]$ is an element in $\BIG_{m + 22}^{\min(19, n + 2)}(\cO_k)$. It follows from (\ref{Subcase3A-e2-main-lem1}) that  that there exists an integer $h \in \bZ$ and an element $\beta \in \cO_k$ such that if $[h, \beta] \in \BIG_m^n(\cO_k)$, then $[ x_1 + x_2\sqrt{\alpha} \; \; y_1 + y_2\sqrt{\alpha} ] $ is an element in $\BIG_{m + 22}^{\min(20, n + 3)}(\cO_k)$

$\bullet$ \textit{Subcase 3B. $\delta = \gcd(s, r) > 1$.}

Write
\begin{align*}
s &= \delta s', \\
r &= \delta r',
\end{align*}
where $\gcd(s', r') = 1$. 
Let
\begin{align*}
x_1 &= sx_1', \\
x_2 &= sx_2',\\
y_1 &= ry_1', \\
y_2 &= ry_2',
\end{align*}
where $\gcd(x_1', x_2') = \gcd(y_1', y_2') = 1$.

We see that
\begin{align}
\label{Subcase3B-e1-main-lem1}
[ x \; \; y ] &= [\delta s'(x_1' + x_2'\sqrt{\alpha}) \; \; \; \delta r'(y_1' + y_2'\sqrt{\alpha})] \nonumber \\
&= \begin{pmatrix} 1 & - 1 \\ 0 & 0 \end{pmatrix} \begin{pmatrix} 1 & 0 \\ 1 - \delta & 0 \end{pmatrix} [s'x_1' + s'x_2'\sqrt{\alpha} \; \; \;  r'y_1' + r'y_2'\sqrt{\alpha}] \nonumber \\
&= \begin{pmatrix} 1 & - 1 \\ 0 & 0 \end{pmatrix} \begin{pmatrix} 1 & 0 \\ 1 - \delta & 0 \end{pmatrix} [x_1'' + x_2''\sqrt{\alpha} \; \; \;  y_1'' + y_2''\sqrt{\alpha}],
\end{align}
where
\begin{align*}
x_1'' &= s'x_1', \\
x_2'' &= s'x_2', \\
y_1'' &= r'y_1', \\
y_2'' &= r'y_2'.
\end{align*}

Note that $\gcd(x_1'', x_2'') = s'$, $\gcd(y_1'', y_2'') = r'$, and $\gcd(s', r') = 1$. So applying the result from \textit{Subcase 3A}, with $x_1, x_2, y_1, y_2$ replaced by $x_1'', x_2'', y_1'', y_2''$, respectively, we deduce that there exists an integer $h \in \bZ$ and an element $\beta \in \cO_k$ such that if $[h, \beta] \in \BIG_m^n(\cO_k)$, then $[ x_1'' + x_2''\sqrt{\alpha} \; \; y_1'' + y_2''\sqrt{\alpha} ] $ is an element in $\BIG_{m + 22}^{\min(20, n + 3)}(\cO_k)$. By Lemmas \ref{l-sub-main-lemma2} and \ref{l-sub-main-lemma1}, and since  $\begin{pmatrix} 1 & - 1 \\ 0 & 0 \end{pmatrix}$, $ \begin{pmatrix} 1 & 0 \\ 1 - \delta & 0 \end{pmatrix}$ are idempotent matrices, we deduce from (\ref{Subcase3B-e1-main-lem1}) that there exists an integer $h \in \bZ$ and an element $\beta \in \cO_k$ such that if $[h, \beta] \in \BIG_m^n(\cO_k)$, then $[ x \; \; y ] $ is an element in $\BIG_{m + 24}^{\min(20, n + 3)}(\cO_k)$.

By all of what we have showed above and Lemma \ref{l-sub-main-lemma2}, there exists an integer $h \in \bZ$ and an element $\beta \in \cO_k$ such that if $[h, \beta] \in \BIG_m^n(\cO_k)$, then $[ x \; \; y ] $ is an element in $\BIG_{m + 24}^{n + 3}(\cO_k)$.

\end{proof}

\begin{lemma}
\label{l-main-lemma3}

Let $k = \bQ(\sqrt{\alpha})$, where $\alpha$ is a positive square-free integer such that $\alpha \equiv 1 \pmod{4}$, and let $\cO_k$ be its ring of integers. Let $x, y$ be elements in $\cO_k$. Then there exist an integer $h \in \bZ$ and an element $\beta \in \cO_k$ such that if $[h, \beta] \in \BIG_m^n(\cO_k)$, then $[ x \; \; y ] $ is an element in $\BIG_{m + 24}^{n + 3}(\cO_k)$

\end{lemma}

\begin{proof}

It is well-known that $\cO_k = \bZ[\dfrac{1 + \sqrt{\alpha}}{2}]$ (see Borevich and Shafarevich \cite{borevich-shafarevich}). Hence each element in $\cO_k$ can be written in the form $a + b\left( \dfrac{1 + \sqrt{\alpha}}{2}\right)$, where $a, b \in \bZ$. Equivalently each element in $\cO_k$ is of the form
\begin{align*}
\dfrac{2a + b + b\sqrt{\alpha}}{2}
\end{align*}
for some integers $a, b$. 

Write $x = \dfrac{2x_1 + x_2 + x_2 \sqrt{\alpha}}{2}$ and $y = \dfrac{2y_1 + y_2 + y_2\sqrt{\alpha}}{2}$, where the $x_i, y_i$ are integers.

Suppose that $x_2 = 0$. Letting $h = x = x_1 \in \bZ$ and $\beta = y \in \cO_k$, we deduce that if $[h, \beta] \in \BIG_m^n(\cO_k)$, then $[ x \; \; y ] $ is an element in $\BIG_{m}^{n}(\cO_k)$, and thus by Lemma \ref{l-sub-main-lemma2}, $[ x \; \; y ] $ is an element in $\BIG_{m + 24}^{n + 3}(\cO_k)$.

Suppose that $y_2 = 0$. We see that
\begin{align*}
[x \; \; y] = [x \; \; y_1] &\longrightarrow_{0_{2, 2}} \begin{pmatrix} 0 & 0 \\ 1 & 0 \end{pmatrix}[-y_1 \; \; x].
\end{align*}

Letting $h = -y_1 \in \bZ$ and $\beta = x \in \cO_k$, we deduce from the above equation that if $[h, \beta] \in \BIG_m^n(\cO_k)$, then $[ x \; \; y ] $ is an element in $\BIG_{m}^{n + 1}(\cO_k)$, and thus by Lemma \ref{l-sub-main-lemma2}, $[ x \; \; y ] $ is an element in $\BIG_{m + 24}^{n + 3}(\cO_k)$.

Suppose that both $x_1, y_1$ are zero. Then 
\begin{align*}
[ x \; \; y ] &= [x_2(1 + \sqrt{\alpha})/2 \; \; y_2(1 + \sqrt{\alpha})/2] \\
&= [(1 + \sqrt{\alpha})/2 \; \; 0][x_2 \; \; y_2] \\
&=\begin{pmatrix} 1 & - 1 \\ 0 & 0 \end{pmatrix} \begin{pmatrix} 1 & 0 \\ (1 - \sqrt{\alpha})/2 & 0 \end{pmatrix} [x_2 \; \; y_2]\\
\end{align*}

Letting $h = x_2 \in \bZ$ and $\beta = y_2 \in \cO_k$, and since $\begin{pmatrix} 1 & - 1 \\ 0 & 0 \end{pmatrix}$, $\begin{pmatrix} 1 & 0 \\ (1 - \sqrt{\alpha})/2 & 0 \end{pmatrix}$ are idempotent matrices, we deduce that if $[h, \beta] \in \BIG_m^n(\cO_k)$, then $[ x \; \; y ] $ is an element in $\BIG_{m + 2}^{n}(\cO_k)$, and thus by Lemma \ref{l-sub-main-lemma2}, $[ x \; \; y ] $ is an element in $\BIG_{m + 24}^{n + 3}(\cO_k)$.

For the rest of the proof, without loss of generality, we can assume that the following assumptions are true:
\begin{itemize}

\item [(i)] both $x_2, y_2$ are nonzero;

\item [(ii)] at least one of $x_1, y_1$ is nonzero.

\end{itemize}

It suffices to consider the following cases.

$\star$ \textit{Case 1. $\gcd(x_1, x_2) = 1$.}

Since $\gcd(x_1, x_2) = 1$, it follows that $\gcd(x_1, x_1 + x_2) = 1$, and thus there exist integers $a_0, b_0$ such that
\begin{align*}
a_0 (x_1 + x_2) + b_0 x_2 = 1.
\end{align*}
Thus
\begin{align}
\label{Case1-e1-main-lemma2}
a(x_1 + x_2) + bx_2 = -y_2,
\end{align}
where $a = -a_0y_2 \in \bZ$ and $b = -b_0y_2 \in \bZ$. 

We see that
\begin{align*}
&[x \; \; y] = [\dfrac{2x_1 + x_2 + x_2 \sqrt{\alpha}}{2} \; \; \dfrac{2y_1 + y_2 + y_2\sqrt{\alpha}}{2}] \\
&\longrightarrow_{(\dfrac{2b + a + a\sqrt{\alpha}}{2})_{1, 2}}[\dfrac{2x_1 + x_2 + x_2 \sqrt{\alpha}}{2} \quad(\dfrac{2b + a + a\sqrt{\alpha}}{2})(\dfrac{2x_1 + x_2 + x_2 \sqrt{\alpha}}{2}) +\dfrac{2y_1 + y_2 + y_2\sqrt{\alpha}}{2}] \\
&= [\dfrac{2x_1 + x_2 + x_2 \sqrt{\alpha}}{2} \qquad bx_1 + \dfrac{bx_2}{2} + \dfrac{ax_1}{2} + \dfrac{ax_2}{4} + \dfrac{ax_2\alpha}{4} + y_1 + \dfrac{y_2}{2} + \dfrac{((2b + a)x_2 + (2x_1 + x_2)a + 2y_2)\sqrt{\alpha}}{4}] \\
&= [\dfrac{2x_1 + x_2 + x_2 \sqrt{\alpha}}{2} \qquad bx_1 + ax_2(\alpha - 1)/4 + y_1] \; (\text{see (\ref{Case1-e1-main-lemma2})}) \\
&\longrightarrow_{0_{2, 2}} \begin{pmatrix} 0 & 0 \\ 1 & 0 \end{pmatrix} [-(bx_1 + ax_2(\alpha - 1)/4 + y_1) \qquad \dfrac{2x_1 + x_2 + x_2 \sqrt{\alpha}}{2}] \\
\end{align*}

By Corollary \ref{c-main3}, $\begin{pmatrix} 0 & 0 \\ 1 & 0 \end{pmatrix}  \in \BIG_{11}^{19}(\cO_k)$. Since $h = -(bx_1 + ax_2(\alpha - 1)/4 + y_1) \in \bZ$ and $\beta = \dfrac{2x_1 + x_2 + x_2 \sqrt{\alpha}}{2}$ is an element in $\cO_k$, we deduce from Lemma \ref{l-sub-main-lemma1} that if $[h, \beta] \in \BIG_m^n(\cO_k)$, then $[x \; \; y]   \in \BIG_{m + 11}^{\min(19, n) + 2}(\cO_k)$.

$\star$ \textit{Case 2. $\gcd(y_1, y_2) = 1$.}

We see that
\begin{align*}
[x \; \; y] &\longrightarrow_{0_{2, 2}} \begin{pmatrix} 0 & 0 \\ 1 & 0 \end{pmatrix}[-y \; \; x] \\
&=  \begin{pmatrix} 0 & 0 \\ 1 & 0 \end{pmatrix}[\dfrac{-2y_1  - y_2 - y_2 \sqrt{\alpha}}{2} \; \; \dfrac{2x_1 + x_2 + x_2\sqrt{\alpha}}{2}].
\end{align*}

Since $\gcd(-y_1, -y_2) = 1$, repeating the same arguments as in \textit{Case 2} of Lemma \ref{l-main-lemma2}, and \textit{Case 1} above, we deduce that there exists an integer $h \in \bZ$ and an element $\beta \in \cO_k$ such that if $[h, \beta] \in \BIG_m^n(\cO_k)$, then $[x \; \; y]$ is an element in $\BIG_{m + 22}^{\min(19, \min(19, n) + 2)}(\cO_k)$ which is equivalent to $[x \; \; y] \in \BIG_{m + 22}^{\min(19, n + 2)}(\cO_k)$.

$\star$ \textit{Case 3. $s = \gcd(x_1, x_2) > 1$ and $r = \gcd(y_1, y_2) >1$.}

$\bullet$ \textit{Subcase 3A. $\gcd(s, r) = 1$.}

Set $\lambda = \gcd(x_1, y_1)$, and $\epsilon = \gcd(x_2, y_2)$. Since $x_2, y_2$ are nonzero, and at least one of $x_1, y_1$ is nonzero, $\lambda, \epsilon$ are positive integers.

Write $x_1 = \lambda z_1$, $y_1 = \lambda w_1$, $x_2 = \epsilon z_2$, and $y_2 = \epsilon w_2$, where $z_1, z_2, w_1, w_2$ are all integers such that $\gcd(z_1, w_1) = \gcd(z_2, w_2) = 1$. 

Assume first that $z_1w_2 - z_2w_1 = 0$. 

We see that $z_1w_2 = z_2w_1$. Since $x_2, y_2$ are nonzero and at least one of $x_1, y_1$ is nonzero, the last identity implies that all of $z_1, z_2, w_1, w_2$ are nonzero.

Since $\gcd(z_1, w_1) = 1$ and $z_1$ divides $z_2w_1$, we deduce that $z_1$ divides $z_2$. On the other hand, since $\gcd(z_2, w_2) = 1$ and $z_2$ divides  $z_1w_2$, it follows that $z_2$ divides $z_1$. Thus $z_1 = \delta z_2$, where $\delta \in \{\pm 1\}$. Thus $w_1 = \delta w_2$. Therefore
\begin{align*}
[\dfrac{2x_1 + x_2 + x_2 \sqrt{\alpha}}{2} \; \; \dfrac{2y_1 + y_2 + y_2\sqrt{\alpha}}{2}] &= [\dfrac{2\lambda z_1 + \epsilon z_2 + \epsilon z_2\sqrt{\alpha}}{2} \qquad  \dfrac{2\lambda w_1 + \epsilon w_2 + \epsilon w_2\sqrt{\alpha}}{2}] \\
&=  [z_2\left(\dfrac{2\lambda \delta + \epsilon  + \epsilon \sqrt{\alpha}}{2}\right) \qquad  w_2\left(\dfrac{2\lambda \delta + \epsilon  + \epsilon \sqrt{\alpha}}{2}\right)] \\
&= [\dfrac{2\lambda \delta + \epsilon  + \epsilon \sqrt{\alpha}}{2} \; \; \; 0][z_2 \; \; w_2].
\end{align*}

We see that $s = \gcd(x_1, x_2) = \gcd(\lambda \delta z_2, \epsilon z_2) = |z_2|\gcd(\lambda \delta, \epsilon)$ and $r = \gcd(y_1, y_2) = \gcd(\lambda \delta w_2, \epsilon w_2) = |w_2|\gcd(\lambda \delta, \epsilon)$. Since $\gcd(s, r) = 1$ and $\gcd(z_2, w_2) = 1$, we deduce that $\gcd(\lambda \delta, \epsilon) = 1$.

Since $\gcd(z_2, w_2) = 1$, Corollary \ref{c-main1} implies that $[z_2 \; \; w_2]$ is in $\BIG_{11}^{18}(\cO_k)$. By the result in \textit{Case 1} with $x_1, x_2$ replaced by $\lambda \delta, \epsilon$, respectively, we deduce that there exist an integer $h \in \bZ$ and an element $\beta \in \cO_k$ such that if $[h \; \; \beta] \in \BIG_m^n(\cO_k)$, then $[\dfrac{2\lambda \delta + \epsilon  + \epsilon \sqrt{\alpha}}{2} \; \; \; 0] \in \BIG_{m + 11}^{\min(19, n) + 2}(\cO_k)$. Using Lemma \ref{l-sub-main-lemma1}, we deduce that there exist an integer $h \in \bZ$ and an element $\beta \in \cO_k$ such that if $[h \; \; \beta] \in \BIG_m^n(\cO_k)$, then $[x \; \; y] \in \BIG_{m + 22}^{\min(18, \min(19, n) + 2)}(\cO_k)$ which is equivalent to $[x \; \; y] \in \BIG_{m + 22}^{\min(18, n + 2)}(\cO_k)$.

Assume now that $z_1w_2 - z_2w_1 \ne 0$. 

Set 
\begin{align*}
\cI = \{\text{primes $\ell$ such that $z_1w_2 - z_2w_1 \equiv 0 \pmod{\ell}$}\}.
\end{align*}
Note that $\cI$ is a finite nonempty set. 

Note that the assumption implies that $\lambda, \epsilon$ are relatively prime, and thus at least one of them is odd. 

Suppose first that $\lambda$ is odd, i.e., every prime factor of $\lambda$ is odd. 

Set
\begin{align*}
\cJ_{\lambda} = \{\text{primes $\ell$ such that $\lambda \equiv 0 \pmod{\ell}$ and $z_2 \not\equiv 0 \pmod{\ell}$}\}.
\end{align*}

Note that all primes in $\cJ_{\lambda}$, if any, are odd. If $\cJ_{\lambda} \ne \emptyset$, write $\cJ_{\lambda} = \cX_{\lambda} \cup \cY_{\lambda}$, where 
\begin{align*}
\cX_{\lambda} = \{\text{primes $\ell$ in $\cJ_{\lambda}$ such that $z_1 \equiv 0 \pmod{\ell}$}\},
\end{align*}
and
\begin{align*}
\cY_{\lambda} = \{\text{primes $\ell$ in $\cJ_{\lambda}$ such that $z_1 \not\equiv 0 \pmod{\ell}$}\}.
\end{align*}

Since $\ell$ is odd for every prime $\ell$ in $\cY_{\lambda}$, one can choose, for each prime $\ell \in \cY_{\lambda}$, an integer $b_{\ell}$ such that $b_{\ell} \not\equiv z_1^{-1}w_1 \pmod{\ell}$ and $b_{\ell} \not\equiv z_2^{-1}w_2 \pmod{\ell}$. For each prime $\ell \in \cX_{\lambda}$, choose an integer $a_{\ell}$ such that $a_{\ell} \not\equiv z_2^{-1}w_2 \pmod{\ell}$. By the Chinese Remainder Theorem, there exists an integer $u_{\lambda}$ such that
\begin{align}
\label{e1-main-lem1}
\begin{cases}
u_{\lambda} \equiv a_{\ell} \pmod{\ell} \; \; \text{for every prime $\ell \in \cX_{\lambda}$}, \\
u_{\lambda} \equiv b_{\ell} \pmod{\ell} \; \; \text{for every prime $\ell \in \cY_{\lambda}$}.
\end{cases}
\end{align}

Note that if exactly one of $\cX_{\lambda}$ and $\cY_{\lambda}$ is empty, $u_{\lambda}$ is chosen so as to satisfy exactly one of the above congruence conditions which corresponds to the nonempty set.  

Set 
\begin{align*}
P_{\lambda} = \prod_{\ell \in \cJ_{\lambda}} \ell.
\end{align*}

If $\cJ_{\lambda} = \emptyset$, we set $P_{\lambda} = u_{\lambda} = 1$. We claim that $\gcd(z_1P_{\lambda}, z_1u_{\lambda} + w_1) = 1$. Since $\gcd(z_1, w_1) = 1$, it is clear that $\gcd(z_1, z_1u_{\lambda} + w_1) = 1$. By the choice of $u_{\lambda}$, it is also clear that $\gcd(P_{\lambda}, z_1u_{\lambda} + w_1) = 1$, and thus
\begin{align*}
\gcd(z_1P_{\lambda}, z_1u_{\lambda} + w_1) = 1.
\end{align*}

By Dirichlet's theorem on primes in arithmetic progressions, there exist infinitely many integers $f$ such that $p = (z_1P_{\lambda})f + z_1u_{\lambda} + w_1 $ is a prime for which $p \not\in \cI$ and $\gcd(p, \epsilon) = 1$. Take such an integer $f$, and set $p = (z_1P_{\lambda})f + z_1u_{\lambda} + w_1 = z_1e + w_1$, where $e = P_{\lambda}f  + u_{\lambda}$. 

We see that
\begin{align}
\label{Subcase3A-e2-main-lem2}
[ \dfrac{2x_1 + x_2 + x_2 \sqrt{\alpha}}{2} \quad \dfrac{2y_1 + y_2 + y_2\sqrt{\alpha}}{2} ] &\longrightarrow_{e_{1, 2}} [\dfrac{2x_1 + x_2 + x_2 \sqrt{\alpha}}{2}  \quad  \dfrac{e(2x_1 + x_2) + 2y_1 + y_2 + (ex_2 + y_2)\sqrt{\alpha}}{2} ] \nonumber\\
&=[\dfrac{2x_1 + x_2 + x_2 \sqrt{\alpha}}{2} \quad  \dfrac{ 2\lambda(ez_1 + w_1) + \epsilon(ez_2 + w_2)     + \epsilon(ez_2 + w_2)\sqrt{\alpha}}{2} ] \nonumber \nonumber\\
&= [\dfrac{2x_1 + x_2 + x_2 \sqrt{\alpha}}{2} \quad  \dfrac{ 2p\lambda + \epsilon(ez_2 + w_2)     + \epsilon(ez_2 + w_2)\sqrt{\alpha}}{2} ] \nonumber \nonumber\\
&=[\dfrac{2x_1 + x_2 + x_2 \sqrt{\alpha}}{2} \quad  \dfrac{2y_1' + y_2'    + y_2'\sqrt{\alpha}}{2} ],
\end{align}
where
\begin{align*}
y_1' &= p\lambda, \\
y_2' &= \epsilon(z_2e + w_2).
\end{align*} 

Using the same arguments as in \textit{Subcase 3A} in the proof of Lemma \ref{l-main-lemma2}, we deduce that $\gcd(y_1', y_2') = 1$. Using the result in \textit{Case 2} for $(y_1', y_2')$ in place of $(y_1, y_2)$, we deduce that there exists an integer $h \in \bZ$ and an element $\beta \in \cO_k$ such that if $[h, \beta] \in \BIG_m^n(\cO_k)$, then $[\dfrac{2x_1 + x_2 + x_2 \sqrt{\alpha}}{2} \quad  \dfrac{2y_1' + y_2'    + y_2'\sqrt{\alpha}}{2} ]$ is an element in $\BIG_{m + 22}^{\min(19, n + 2)}(\cO_k)$. It follows from (\ref{Subcase3A-e2-main-lem2}) that  that there exists an integer $h \in \bZ$ and an element $\beta \in \cO_k$ such that if $[h, \beta] \in \BIG_m^n(\cO_k)$, then $[x \; \; y] $ is an element in $\BIG_{m + 22}^{\min(20, n + 3)}(\cO_k)$.

Suppose now that $\epsilon$ is odd, i.e., every prime factor of $\epsilon$ is odd. We use a similar argument as above with $\epsilon$ in place of $\lambda$ to deduce the lemma. Indeed define $\cJ_{\epsilon}$, $\cX_{\epsilon}$, $\cY_{\epsilon}$, $u_{\epsilon}$, $P_{\epsilon}$, $p$, and $e$ as in \textit{Subcase 3A} of the proof of Lemma \ref{l-main-lemma2}. Recall that $e = P_{\epsilon}f  + u_{\epsilon}$, and $p = (z_2P_{\epsilon})f + z_2u_{\epsilon} + w_2 = z_2e + w_2$ for some integer $f$ such that $p$ is a prime. 

We see that
\begin{align}
\label{Subcase3A-e3-main-lem2}
[ \dfrac{2x_1 + x_2 + x_2 \sqrt{\alpha}}{2} \quad \dfrac{2y_1 + y_2 + y_2\sqrt{\alpha}}{2} ] &\longrightarrow_{e_{1, 2}} [\dfrac{2x_1 + x_2 + x_2 \sqrt{\alpha}}{2}  \quad  \dfrac{e(2x_1 + x_2) + 2y_1 + y_2 + (ex_2 + y_2)\sqrt{\alpha}}{2} ] \nonumber\\
&=[\dfrac{2x_1 + x_2 + x_2 \sqrt{\alpha}}{2} \quad  \dfrac{ 2\lambda(ez_1 + w_1) + \epsilon(ez_2 + w_2)     + \epsilon(ez_2 + w_2)\sqrt{\alpha}}{2} ] \nonumber \\
&= [\dfrac{2x_1 + x_2 + x_2 \sqrt{\alpha}}{2} \quad  \dfrac{ 2\lambda(ez_1 + w_1) + p\epsilon     + p\epsilon\sqrt{\alpha}}{2} ] \nonumber \\
&=[\dfrac{2x_1 + x_2 + x_2 \sqrt{\alpha}}{2} \quad  \dfrac{2y_1' + y_2'    + y_2'\sqrt{\alpha}}{2} ],
\end{align}
where
\begin{align*}
y_1' &=  \lambda(z_1e + w_1), \\
y_2' &= p\epsilon.
\end{align*} 

Following the same arguments as in the last part of \textit{Subcase 3A} in the proof of Lemma \ref{l-main-lemma2}, one can prove that $\gcd(y_1', y_2') = 1$. Thus using (\ref{Subcase3A-e3-main-lem2}) and \textit{Case 2} above, there exists an integer $h \in \bZ$ and an element $\beta \in \cO_k$ such that if $[h, \beta] \in \BIG_m^n(\cO_k)$, then $[x \; \; y] $ is an element in $\BIG_{m + 22}^{\min(20, n + 3)}(\cO_k)$.

$\bullet$ \textit{Subcase 3B. $\delta = \gcd(s, r) > 1$.}

Write
\begin{align*}
s &= \delta s', \\
r &= \delta r',
\end{align*}
where $s', r'$ are integers such that $\gcd(s', r') = 1$. 
Let
\begin{align*}
x_1 &= sx_1', \\
x_2 &= sx_2',\\
y_1 &= ry_1', \\
y_2 &= ry_2',
\end{align*}
where $x_1', x_2', y_1', y_2'$ are integers such that $\gcd(x_1', x_2') = \gcd(y_1', y_2') = 1$.

We see that
\begin{align}
\label{Subcase3B-e1-main-lem3}
[ x \; \; y ] &= [\dfrac{2x_1 + x_2 + x_2 \sqrt{\alpha}}{2} \quad \dfrac{2y_1 + y_2 + y_2\sqrt{\alpha}}{2}] \nonumber \\
 &=[\dfrac{\delta s'(2x_1' + x_2' + x_2'\sqrt{\alpha})}{2} \quad \dfrac{\delta r'(2y_1' + y_2' + y_2'\sqrt{\alpha})}{2}] \nonumber \\
&=\begin{pmatrix} 1 & - 1 \\ 0 & 0 \end{pmatrix} \begin{pmatrix} 1 & 0 \\ 1 - \delta & 0 \end{pmatrix} [\dfrac{s'(2x_1' + x_2' + x_2'\sqrt{\alpha})}{2} \quad \dfrac{r'(2y_1' + y_2' + y_2'\sqrt{\alpha})}{2}] \nonumber \\
&= \begin{pmatrix} 1 & - 1 \\ 0 & 0 \end{pmatrix}\begin{pmatrix} 1 & 0 \\ 1 - \delta & 0 \end{pmatrix} [\dfrac{2x_1'' + x_2'' + x_2'' \sqrt{\alpha}}{2} \quad \dfrac{2y_1'' + y_2'' + y_2''\sqrt{\alpha}}{2}],
\end{align}
where 
\begin{align*}
x_1'' &= s'x_1', \\
x_2'' &= s'x_2', \\
y_1'' &= r'y_1', \\
y_2'' &= r'y_2'.
\end{align*}

Note that $\gcd(x_1'', x_2'') = s'$, $\gcd(y_1'', y_2'') = r'$, and $\gcd(r', s') = 1$. So applying the result from \textit{Subcase 3A}, with $x_1, x_2, y_1, y_2$ replaced by $x_1'', x_2'', y_1'', y_2''$, respectively, we deduce that there exists an integer $h \in \bZ$ and an element $\beta \in \cO_k$ such that if $[h, \beta] \in \BIG_m^n(\cO_k)$, then $[\dfrac{2x_1'' + x_2'' + x_2'' \sqrt{\alpha}}{2} \quad \dfrac{2y_1'' + y_2'' + y_2''\sqrt{\alpha}}{2}]$ is an element in $\BIG_{m + 22}^{\min(20, n + 3)}(\cO_k)$. By Lemmas \ref{l-sub-main-lemma2} and \ref{l-sub-main-lemma1}, and since  $\begin{pmatrix} 1 & - 1 \\ 0 & 0 \end{pmatrix}$, $ \begin{pmatrix} 1 & 0 \\ 1 - \delta & 0 \end{pmatrix}$ are idempotent matrices, we deduce from (\ref{Subcase3B-e1-main-lem3}) that there exists an integer $h \in \bZ$ and an element $\beta \in \cO_k$ such that if $[h, \beta] \in \BIG_m^n(\cO_k)$, then $[ x \; \; y ] $ is an element in $\BIG_{m + 24}^{\min(20, n + 3)}(\cO_k)$.

By all of what we have showed above and Lemma \ref{l-sub-main-lemma2}, there exists an integer $h \in \bZ$ and an element $\beta \in \cO_k$ such that if $[h, \beta] \in \BIG_m^n(\cO_k)$, then $[ x \; \; y ] $ is an element in $\BIG_{m + 24}^{n + 3}(\cO_k)$.

\end{proof}

The following lemma is a slightly modified version of Theorem {\bf3.1} in Cossu and Zanardo \cite{cossu-zanardo}.

\begin{lemma}
\label{l-main-lemma4}

Let $\cO_k$ be the ring of integers of a real quadratic field $k = \bQ(\sqrt{\alpha})$, where $\alpha$ is a positive square-free integer. Let $x, y$ be elements in $\cO_k$ such that $x\cO_k + y\cO_k = z\cO_k$ for some nonzero element $z \in \cO_k$, i.e. $x, y$ generates the principal ideal of $z\cO_k$. Then $[x \; \; y] \in \BIG_{13}^{18}(\cO_k)$.

\end{lemma}

\begin{proof}

By assumption, there exist elements $x_1, y_1 \in \cO_k$ such that $x = x_1z$ and $y = y_1z$. Since $x\cO_k + y\cO_k = z\cO_k$, there exist elements $a, b \in \cO_k$ such that
\begin{align*}
(x_1z)a + (y_1z)b = z, 
\end{align*}
and thus 
\begin{align*}
x_1a + y_1b = 1.
\end{align*}
Thus $\begin{pmatrix} x_1 & y_1 \\ -a & b \end{pmatrix} \in \SL_2(\cO_k)$. By Corollary \ref{c-main1}, $[x_1 \; \; y_1] \in \BIG_{11}^{18}(\cO_k)$. Since
\begin{align*}
[x \; \; y] = \begin{pmatrix} 1 & - 1\\ 0 & 0 \end{pmatrix} \begin{pmatrix} 1 & 0 \\ 1 - z & 0 \end{pmatrix}[x_1 \; \; y_1],
\end{align*}
and $\begin{pmatrix} 1 & - 1\\ 0 & 0 \end{pmatrix}$,  $\begin{pmatrix} 1 & 0 \\ 1 - z & 0 \end{pmatrix}$ are idempotent matrices, we deduce from Lemmas \ref{l-sub-main-lemma0} and \ref{l-sub-main-lemma1} that $[x \; \; y] \in \BIG_{13}^{18}(\cO_k)$, which proves the lemma.

\end{proof}

\begin{theorem}
\label{thm-main}

Let $k = \bQ(\sqrt{\alpha}$ be a real quadratic number field,  where $\alpha$ is a positive square-free integer. Let $\cO_k$ be the ring of integers of $k$. Let $\cM$ be the set of $2\times 2$ matrices over $\cO_k$ of the form $\begin{pmatrix} x & y \\ 0 & 0 \end{pmatrix}$, where $x, y$ are elements in $\cO_k$. Then $\cM$ is a subset of $\BIG_{15}^{19}(\cO_k)$, i.e., every matrix in $\cM$ belongs to $\BIG_{15}^{19}(\cO_k)$.

\end{theorem}

\begin{proof}

Throughout the proof, for each prime $p$, we denote by $v_p$ the \textit{$p$-adic valuation} on $\bQ$. 

By Lemmas \ref{l-main-lemma2} and \ref{l-main-lemma3}, it suffices to prove that the subset $\cM_0$ of $\cM$ consisting of matrices of the form $[x \; \; y]$, where $x \in \bZ$ and $y \in \cO$ is a subset of $\BIG(\cO_k)$. In order to prove this, we will use the techniques in the proof of Theorem \textbf{3.2} in Cossu and Zanardo \cite{cossu-zanardo}.

Suppose first that there exists a non-unit element $z \in \cO_k$ such that $x = x_1z$ and $y = y_1z$, where $x_1, y_1$ are elements in $\cO_k$ such that $x_1, y_1$ have no common non-unit factors in $\cO_k$. Then
\begin{align*}
[x \; \; y] = \begin{pmatrix} 1 & - 1\\ 0 & 0 \end{pmatrix} \begin{pmatrix} 1 & 0 \\ 1 - z & 0 \end{pmatrix}[x_1 \; \; y_1].
\end{align*}
Since $\begin{pmatrix} 1 & - 1\\ 0 & 0 \end{pmatrix}$,  $\begin{pmatrix} 1 & 0 \\ 1 - z & 0 \end{pmatrix}$ are idempotent matrices, we deduce from Lemmas \ref{l-sub-main-lemma1} and \ref{l-sub-main-lemma2} that if $[x_1 \; \; y_1] \in \BIG_{13}^{19}(\cO_k)$ then $[x \; \; y] \in \BIG_{15}^{19}(\cO_k)$. Thus it suffices to show that if $x, y$ have no common non-unit factors in $\cO_k$, then $[x \; \; y] \in \BIG_{13}^{19}(\cO_k)$. 

On the other hand, note that Lemma \ref{l-main-lemma4} implies that if $x\cO_k + y\cO_k$ is a principal ideal in $\cO_k$, then $[x \; \; y] \in \BIG_{13}^{18}(\cO_k)$. By Lemma \ref{l-sub-main-lemma2}, we deduce that $[x \; \; y] \in \BIG_{15}^{19}(\cO_k)$ if $x\cO_k + y\cO_k$ is a principal ideal in $\cO_k$.

So without loss of generality, for the rest of the proof, we can further assume that the following are true:
\begin{itemize}

\item [(i)] $x, y$ have no common non-unit factors in $\cO_k$;

\item [(ii)]  $x\cO_k + y\cO_k$ is not a principal ideal; especially $x\cO_k + y\cO_k \ne \cO_k$, which implies that 
\begin{align}
\label{e1-main-thm}
m = \gcd(x, ||y||) \ne 1,
\end{align}
where for the rest of this paper, $||y||$ denotes the norm of $y$ in $\cO_k$, i.e., $||y|| = y \bar{y}$, where $\bar{y}$ is the conjugate element of $y$ (see \cite{borevich-shafarevich}). 
\end{itemize}

Our aim is to show that if conditions (i) and (ii) are satisfied, then $x \; \; y] \in \BIG_{13}^{19}(\cO_k)$. We consider the following cases.

\textit{Case 1. $s = \gcd(x, \dfrac{||y||}{m}) = 1$.}

Following the same arguments as in Step 1 of the proof of Theorem {\bf3.2} in Cossu and Zanardo \cite{cossu-zanardo}, one can write
\begin{align*}
[x \; \; y] = [x' \; \; y']\begin{pmatrix} a & b \\ c & 1 - a\end{pmatrix},
\end{align*}
where $a, b, c \in \cO_k$ such that $\begin{pmatrix} a & b \\ c & 1 - a\end{pmatrix}$ is an idempotent matrix, and $x', y' \in \cO_k$ such that $\begin{pmatrix} x' & y' \\ u & v\end{pmatrix} \in \SL_2(\cO_k)$ for some elements $u, v \in \cO_k$. By Corollary \ref{c-main1}, $[x' \; \; y'] \in \BIG_{11}^{18}(\cO_k)$, and it thus follows from Lemmas \ref{l-sub-main-lemma1} and \ref{l-sub-main-lemma2} that $[x \; \; y] \in \BIG_{12}^{18}(\cO_k)$.

\textit{Case 2. $s = \gcd(x, \dfrac{||y||}{m}) \ne 1$.}

In this case, we consider the following subcases.

\textit{Subcase 2A. $\alpha \equiv 2, 3 \pmod{4}$.}

In this subcase, $\cO = \bZ[\sqrt{\alpha}]$, and each element in $\cO$ can be written in the form $a + b\sqrt{\alpha}$ for some integers $a, b \in \bZ$. 

Write $y = y_1 + y_2\sqrt{\alpha}$, where $y_1, y_2$ are integers. By assumption, we know that $x, y$ have no common non-unit factors in $\cO_k$, and thus $\gcd(x, y_1, y_2) = 1$. One can write
\begin{align*}
x &= x_0m, \\
||y|| &= \lambda m,
\end{align*}
where $x_0$ and $\lambda$ are integers such that $\gcd(x_0, \lambda) = 1$. 

By \textit{Fact 2} in the proof of Theorem {\bf3.2} in \cite{cossu-zanardo}, there exists an integer $e \in \bZ$ such that 
\begin{align}
\label{e2-main-thm}
\gcd(x, ||y + ex||/m) = 1.
\end{align}

By computation, we see that
\begin{align*}
||y + ex|| = m(\lambda + 2x_0ey_1 + mx_0^2e^2) = ||y|| + x2ey_1 + x^2e^2.
\end{align*}

Set $\gamma = \gcd(x, ||y + ex||)$. Since $x = mx_0$, we see from the above equation that $m$ divides $\gamma$. By (\ref{e2-main-thm}), there exist integers $a, b$ such that 
\begin{align*}
ax + b(||y + ex||/m) = 1,
\end{align*}
and thus
\begin{align*}
(am)x + b||y + ex|| = m.
\end{align*}
Thus $\gamma$ divides $m$, and therefore $m = \gamma = \gcd(x, ||y + ex||)$. Using the results from \textit{Case 1} with $x, y + ex$ in the roles of $x, y$, respectively, we deduce that $[x \; \; y + ex] \in \BIG_{12}^{18}(\cO_k)$. Since
\begin{align*}
[x \; \; y]^{e_{1, 2}} = [x \; \; y + ex],
\end{align*}
we deduce that $[x \; \; y] \in \BIG_{12}^{19}(\cO_k)$.

\textit{Subcase 2B. $\alpha \equiv 1 \pmod{4}$.}

In this subcase, $\cO = \bZ[(1 + \sqrt{\alpha})/2]$, and each element in $\cO$ can be written in the form $a + b\sqrt{\alpha}$ for some integers $a, b \in \bZ$ with $a \equiv b \pmod{2}$. 

By \textit{Facts 2(a) and 2(b)} in \textit{Step 3} of the proof in Theorem {\bf3.2} in \cite{cossu-zanardo}, there exists an integer $e \in \bZ$ such that 
\begin{align}
\label{e3-main-thm}
\gcd(x, ||y + ex||/m) = 1.
\end{align}

Using (\ref{e3-main-thm}, and the same arguments as in \textit{Subcase 2A}, we deduce that $[x \; \; y] \in \BIG_{12}^{19}(\cO_k)$.

By what we have verified in \textit{Cases 1 and 2}, it follows from Lemma \ref{l-sub-main-lemma2} that if $x, y$ are elements in $\cO_k$ that satisfy conditions (i) and (ii) above, then $[x \; \; y] \in \BIG_{13}^{19}(\cO_k)$. By the discussion at the beginning of the proof, the theorem follows immediately.

\end{proof}

\end{document}